\documentclass{amsart}
\usepackage{amsmath,amssymb,latexsym}[13 January 2022]
\usepackage{amscd}
\newtheorem{proposition}{Proposition}[section]
\newtheorem{example}[proposition]{Example}
\newtheorem{lemma}[proposition]{Lemma}
\newtheorem{corollary}[proposition]{Corollary}
\newtheorem{theorem}[proposition]{Theorem}
\newtheorem{remark}[proposition]{Remark}

\newcommand{\dpt}{\operatorname{depth}}
\newcommand{\hgt}{\operatorname{height}}
\newcommand{\rank}{\operatorname{rank}}

\newcommand{\spn}{\operatorname{span}}
\newcommand{\gen}{\operatorname{gen}}
\newcommand{\codim}{\operatorname{codim}}

\newcommand{\im}{\operatorname{Im}}

\newcommand{\Der}{\operatorname{Der}}
\newcommand{\Vect}{\operatorname{Vect}}

\newcommand{\La}{\Lambda}
\newcommand{\Om}{\Omega}

\def \R     {\Bbb R}
\def \C     {\Bbb C}

\def \AA    {{\mathcal A}}

\def \II    {{\mathcal I}}
\def \JJ    {{\mathcal J}}

\def \OO    {{\mathcal O}}
\def \a      {\alpha}
\def \b      {\beta}
\def \O     {\Omega}
\def \o     {\omega}
\def \g     {\gamma}
\def \w     {\wedge}
\def \lra   {\longrightarrow}

\def \nzd  {non-zero-divisor\ }
\def \nzds  {non-zero-divisors\ }
\def \Rloc {R_{[a]}}
\def \Mloc {M_{[a]}}
\def \Iang {\langle I\rangle}
\def \Jang {\langle J\rangle}
\def \Kang {\langle K\rangle}
\def \po {\partial_{\omega}}

\def \tg  {\tilde\gamma}

\def \bR  {\bar R}
\def \bM {\bar M}
\def \bK {\bar K}

\def \bS {\bar S}
\def \ba {\bar a}
\def \bb  {\bar b}
\def \bc  {\bar c}
\def \be  {\bar e}
\def \bo  {\bar\omega}
\def \bg {\bar\gamma}
\def \bO  {\bar\Omega}
\def \ba  {\bar a}

\def \hJ {\hat J}
\def \hf {\hat f}
\def \hg {\hat\gamma}
\def \fracR {R_{fr}}
\def \fracM {M_{fr}}
\def \fracK {K_{fr}}
\def \Res {\mathcal{R}es}

\def \owedge {\bigcirc\!\!\!\!\!\!\wedge\ }

\begin{document}

\title{Division properties in exterior algebras of free modules\\
and logarithmic residua}
\author{B. Jakubczyk} 
\address{Institute of Mathematics, Polish Academy of Sciences\\
00-656 Warsaw, Śniadeckich 8, Poland\\
E-mail: b.jakubczyk@impan.pl}
\keywords{Exterior algebra over free module, divisibility, depth, logarithmic residua.}

\begin{abstract}
Let $M$ be a free module of rank $m$ over a commutative unital ring $R$ and let $N$ be its free submodule. We consider the problem when a given element of the exterior product $\Lambda^pM$ is divisible, in a sense, over elements of the exterior product $\Lambda^r N$, $r\le p$. Precisely, we give conditions under which an element $\eta\in\Lambda^pM$ can be expressed as a finite sum of skew-products of elements of $\Lambda^r N$ and elements of $\Lambda^{p-r} M$. For a given basis $\omega_1,\dots,\omega_k$ in $N$ the elements of  $\Lambda^{p-r} M$ are unique in a specified sense. Necessary and sufficient conditions for such divisibility take a simple form, provided that the submodule is embedded in $M$ with singularities having the depth larger then $p-r+1$.  In the special case where $r=k=rank N$ the divisibility property means that $\eta=\Omega\wedge\gamma$ where $\Omega=\omega_1\wedge\cdots\wedge\omega_k$  and $\gamma\in\Lambda^{p-k}M$.

More detailed statements of these results are then used to state criteria for existence and uniqueness of algebraic logarithmic residua when the ``divisor'' is defined by elements $f_1,\dots,f_k\in R$. Special cases are multidimensional logarithmic residua in complex analysis.
\end{abstract}    
\maketitle
	
\section{Introduction}

Consider a commutative unital ring  $R$ and let $M$ be a free $R$-module of finite rank $m$. We will denote by $\La^q M$ the $q$th exterior product of $M$ and $\La M=\oplus_{q=0}^m \La^q M$ will be the exterior algebra generated by $M$.  
We shall consider division problems in the algebra $\La M$ related to the following elementary questions.

Consider two elements $\eta\in \La^pM$ and $\a\in \La^rM$. If $p\ge r$ then it is natural to ask

\medskip
{\bf Question 1.} {\it When $\eta$ is divisible over $\a$, i.e., there exists  $\g\in\La^{p-r}M$ such that 
\[
\eta=\a\w\g\ ?
\]}
This question is largely open for general $R$. In the case where  $R$ is the field of real numbers $\R$ and $M$ is the vector space $\R^m$ the following general result of Dacorogna and Kneuss \cite{DK} answers the question. There exists such $\g$ if and only if $\ker \a\subset \ker \eta$, where one defines  $\ker \a=\{\b\in \La^{m-p}\R^m|\a\w\b=0\}$ and  $\ker \eta=\{\b\in \La^{m-p}\R^m|\eta\w\b=0\}$.

For general $R$ we will answer Question 1 in the special case when $\a$ is a product of $\o_i\in M$. Namely, let $\o_1,\dots,\o_k$ be fixed elements of $M$, where $1\le k\le m$. We will answer

\medskip
{\bf Question 2.} {\it When there exists $\g\in \La^{p-k}M$ such that 
\[
\eta=\o_1\w\cdots\w\o_k\w\g?
\]}
Question 2, similarly as two other ones stated below, is nontrivial when the submodule of $M$ generated by the elements $\o_1,\dots,\o_k$ is singularly embedded in $M$ which means that the elements are linearly independent but can not be completed to a basis of $M$.

We can also ask if for a given $\eta\in M$ it is possible to find elements $a_i\in R$ such that $\eta=\sum_i a_i\o_i$. More generally, 

\medskip
{\bf Question 3.} {\it Given $\eta\in \La^pM$, is it possible to represent $\eta$ as
\[
\eta=\sum_i\o_i\w\g_i,
\]
with some $\g_i\in \La^{p-1}M$?} 
\medskip

An answer to this question was given in the main theorem of \cite{S} for Noetherian rings $R$. Questions 2 and 3 are particular cases of the following main problem considered in the paper.

\medskip
{\bf Question 4.} {\it
Given $\eta\in\La^pM$ and $1\le r\le k$ such that $p\ge r$, when one can write
\[
\eta=\sum \o_J\w\g_J \eqno{(R)}
\]
with some $\g_J\in \La^{p-r}M$, where $\o_J\in \La^rM$ are of the form $\o_J=\o_{j_1}\w\cdots\w\o_{j_r}$?}
\medskip

Such questions appear in differential geometry when one considers k-tuples of differential 1-forms or vector fields on a manifold $N$. Singularities of such tuples (points where they are linearly dependent) are unavoidable if the topology of the manifold is not sufficiently trivial, see \cite{K}. Already for $k=1$ and nonzero Euler characteristic of $N$ every smooth vector field or differential 1-form on $N$ has a singular point. Also any holomorphic 1-form on a smooth complex projective variety of general type must vanish at some point \cite{PS}. Answers to above questions may help to deal with such local singularities (as was the case in \cite{JZ}), in which case $R$ will be the ring of germs of functions on manifolds in the holomorphic, real analytic, or $C^\infty$ category (where $R$ is non-Noetherian).

Our first main result, statement (i) in Theorem \ref{thm1-existence} in the next section, will give a criterium for $\eta$ to be representable in the form (R). To have such a representation we will assume that $p-r+2<d$, where $d$ is the depth of the ideal $I(\O)$ generated by the coefficients of the k-form 
\[
\O=\o_1\w\cdots\w\o_k.
\] 
Without this assumption an analogous representation holds for $\eta$ replaced by $a^n\eta$ with $a$ an arbitrary element of the ideal $I(\O)$ and $n$ sufficiently large (statement (ii) of Theorem \ref{thm1-existence}). The proof of Theorem \ref{thm1-existence} is postponed to Section 6. A uniqueness property of the representation is stated in Theorem \ref{thm1-uniqueness}. Finally, an equivalent statement of Theorem  \ref{thm1-existence}, formulated as a property of a (singularly embedded) free submodule of $M$, is given in  Theorem \ref{thm2}. In the Appendix we recall properties of the depth used in stating some of the results.

In Section 3 we state direct consequences of results from Section 2 in the case where, in addition to the earlier data, we are given a finitely generated ideal $\II\subset R$. In this case we define residua as results of the division discussed in Secton 2 factorized via the submodule $\II M$ and state a general theorem on their existence and uniqueness. 

More specific residua, called logarithmic, are defined in Section 4 where $M$ is the module of derivations of $R$ and the elements $\o_i\in M$ are replaced with the algebraic differentials $df_i\in M^*$ of generators $f_1,\dots,f_k$ of $\II$. Finally, in Section 5 we translate the theorem on algebraic logarithmic residua to a theorem on multidimensional residua in complex analysis. 

We do not assume that the ring $R$ and the module $M$ are Noetherian which makes the proofs more delicate. 

\section{Main results}\label{Main results}

Let $R$ denote a commutative ring with unity and let $M$ be a free module over $R$ of finite rank $m$. We will denote its $q$th exterior product by $\La^qM$, with identifications $\La^0M=R$, $\La^1M=M$, and $\La^qM=0$ for $q<0$ and $q>m$. Note that, choosing a basis $e_1,\dots,e_m$ in $M$ and $q\in\{1,\dots,m\}$, we have a natural basis in $\La^qM$ which consists of exterior products
\[
e_I=e_{i_1}\w\cdots\w e_{i_q}, \ \ 1\le i_1<\cdots<i_q\le m. \eqno{(E)}
\]
Elements $\eta\in\La^qM$ can be written as $\eta=\sum_Ia_Ie_I$, with unique coefficients $a_I\in R$, $I=i_1\cdots i_q$.

Consider a sequence of elements
\[
\o_1,\dots,\o_k\in M,
\] 
$1\le k\le m$, which will be fixed throughout the section.  We will use the following notation.  For a given $0\le r\le k$ we will denote by $\JJ(r,k)$ the set of \emph{strictly monotone multiindices} 
\[
\JJ(r,k)=\{J=j_1\cdots j_r:\,1\le j_1<\cdots<j_r\le k\},
\]
including the empty one $J=\emptyset$ with $r=0$. Sometimes, when $r$ is not fixed, it will be convenient to denote the length $r$ of $J$ by $\Jang$.  Given $J\in\JJ(r,k)$, we will denote
\[
\o_J=\o_{j_1}\w\cdots\w\o_{j_r}. 
\]
where $\o_J=1$ when $J=\emptyset$. 

Denote $\O=\o_1\w\cdots\w\o_k$. Using the basis (E) with $q=k$ we can write $\Om=\sum_I a_Ie_I$ where $a_I\in R$ and the sum is taken over $I\in \JJ(k,m)$. The ideal generated by the coefficients $a_I$,
\[
I(\Om)=\spn_R\{a_I, I\in \JJ(k,m)\},
\]
is independent of the choice of the basis in $M$ and will play a crutial role in our considerations. 

Recall that the depth of a proper ideal $I\subset R$, denoted $\dpt I$, is the largest length of a regular sequence in $I$ (see Appendix for a brief recall of properties of regular sequences). In Noetherian rings we have $\dpt I\le \hgt I\le \gen I$ (in Cohen-Macaulay rings $\dpt I=\hgt I$), where $\gen I$ denotes the minimal number of generators of the ideal $I$.  Additionally, one defines $\dpt R=\infty$. Recall that $\rank M=m$.

\begin{theorem}[Existence]\label{thm1-existence}
Fix integers $1\le p,r,s\le m$ such that $p\ge r$ and $r+s=k+1$.  

(i) Assume  that
\[
p-r\le \dpt I(\Om)-2.  \eqno{(DC)}
\]
Then an element  $\eta\in \La^p M$ can be represented in the form
\[
\eta=\sum_{J\in \JJ(r,k)}\o_J\w\g_J, \ \ \text{for\ some}\  \ \g_J\in \La^{p-r}M, \eqno{(A)}
\]
if and only if
\[
\o_I\w\eta=0,\ \ \text{for\ all}\ \ I\in \JJ(s,k).\eqno{(B)}
\] 

(ii) Without assuming the depth condition (DC), for any $\eta\in \La^pM$ satisfying (B) there exists $n>0$ (not depending on $\eta$ if the module of $\eta$ satisfying (B) is finitely generated) such that for $b=a^n$, with arbitrary $a\in I(\Omega)$, we have 
\[
b\eta=\sum_{J\in\JJ(r,k)} \o_J\w\g_J, \ \ \text{for\ some}\  \ \g_J\in \La^{p-r}M. \eqno{(A')}
\]
Conversely, if (A') holds for some \nzd $b\in R$ then (B) holds, too.
\end{theorem}

\begin{remark}\label{rem1}
{\rm Implication (A) $\Rightarrow$ (B) is trivial and holds for any $0\le p\le m$ without assuming condition (DC). Namely, multiplying both sides of (A) by $\o_I$ with $I\in \JJ(s,k)$ gives zero on the right hand side as each product $\o_I\w\o_J$, with $J\in \JJ(r,k)$, contains twice some $\o_i$ due to $r+s=k+1$. For the same reason condition (A') implies that $b\eta\w\o_I=0$ for all $I\in \JJ(s,k)$. Thus, if (A') holds for a \nzd $b\in R$ then (B) holds, too. Note that both statements are empty or trivial if $\O=0$.}
\end{remark}

\begin{remark}\label{rem-depth}\rm
For Noetherian $R$ we have $d:=\dpt I(\Om)\le m-k+1$, if $I(\O)\not=R$. Indeed, $I(\Om)$ is generated by the $k\times k$ minors of the $k\times m$ matrix of the coefficients of the forms $\o_1,\dots,\o_k$. Let $I$ be the ideal generated by the first $m-k+1$ coefficients of $\o_1$. We see that any $k\times k$ minor has at least one factor among these coefficients which implies that $I(\O)\subset I$. Thus we have $d\le\dpt I$. The conclusion follows from $\dpt I\le\gen I$, since $R$ is Noetherian, and $\gen I\le m-k+1$.
\end{remark}

For $r=1$ the first statement in the theorem reduces to the ensuing known facts.

\begin{corollary}\label{cor1}
(i) If $\dpt I(\O)\ge 2$ then any $\eta\in M$ satisfying $\O\w\eta=0$ can be represented as $\eta=\sum a_i\o_i$ with $a_i\in R$.

(ii) More generally, for any $\eta\in\La^pM$ with $p<\dpt I(\O)$ the equality $\O\w\eta=0$ implies that $\eta=\sum \o_i\w\g_i$ for some $\g_i\in\La^{p-1}M$ (answer to Question 3).

(iii) Let $k=1$ and denote $\o_1=\o$. Then for any $\eta\in\La^pM$, with $p<\dpt I(\o)$, the equality $\o\w\eta=0$ implies that  $\eta=\o\w\g$ for some $\g\in \La^{p-1}M$.
\end{corollary}

The second statement was proved in Saito \cite{S} assuming $R$ was Noetherian and in \cite{J} without this assumption. The third statement just means exactness of the Koszul complex 
\[
M\mathop{\lra}^{\po} \La^2M\mathop{\lra}^{\po}\cdots\mathop{\lra}^{\po}\La^{d-1}M\mathop{\lra}^{\po}\La^dM,
\]
where $d=\dpt I(\o)$ and $\po$ is the operator of exterior multiplication by $\o$, $\po:\a\mapsto\o\w\a$ (cf. \cite{E}). Exactness of this complex is used in many contexts but it is usually assumed that $R$ is Noetherian.

For $r=k$  Theorem \ref{thm1-existence} gives the following answer to Question 2.

\begin{corollary}\label{cor2}
Let $\eta\in \La^pM$, where $k\le p\le k+\dpt I(\O)-2$. Then $\eta=\o_1\w\cdots\w\o_k\w\g$ for some $\g\in \La^{p-k}M$ if and only if  $\eta\w\o_i=0$ for $i=1,\dots,k$.
\end{corollary}

If $p+s>m$ then $\o_J\w\eta\in \La^{p+s}M=0$ and condition (B) is automatically satisfied. Thus, 

\begin{corollary}\label{cor3}
If $p>m-s$ (equivalently, $p-r\ge m-k$) then $\eta$ satisfies condition (A'), i.e. (A') holds for $a^n\eta$ with any $\eta\in\La^pM$ and $a\in I(\O)$, for some $n>0$ independent of $\eta$ and $a$ .
\end{corollary}

Note that the depth condition fails in this case since $m-k\ge d-1$, by Remark \ref{rem-depth}.

\begin{example}\rm
To illustrate the above corollary take $m=5$, $k=4$, $r=2$, $s=3$, $p=3$ and $\o_1,\dots,\o_4\in M$. Then for any $\eta\in \La^3M$ there exists $n\ge 1$ such that for arbitrary $a\in I(\o_1\w\cdots\w\o_4)$ one can write $a^n\eta=\sum_{i<j}\o_i\w\o_j\w\g_{ij}$ for some $\g_{ij}\in M$.
\end{example} 

Uniqueness of the representations of $\eta$ in Theorem \ref{thm1-existence} is explained by the following result. Let $M^*$ denote the dual module to $M$. Note that $\g\in\La^qM$ can be treated as a $q$th skew symmetric form on $M^*$.  Each $\o_i$ has a well defined kernel $\ker \o_i=\{g\in M^*| g(\o_i)=0\}$ and their comon kernel will be denoted
\[
K=\ker \o_1\cap\cdots\cap \ker\o_k\subset M^*.
\]

\begin{theorem}[Uniqueness]\label{thm1-uniqueness}
Let  $\rank M\ge p\ge r\ge 1$ and $1\le k\le m$. If the ideal $I(\O)$ contains a \nzd then the following holds.

(i) If $\eta\in\La^pM$ is given in the form (A) then the elements $\g_J$, $J\in\JJ(r,k)$, are unique when restricted to $K$ which means that $\g_J(\a_1,\dots,\a_{p-r})$ are determined by $\eta$ and $\a_1,\dots,\a_{p-r}\in K$.

(ii) More generally, if $b$ is a \nzd in $R$ and $b\eta=\sum_{J\in\JJ(r,k)}\o_J\w\g_J$ then the elements $\g_J/b$ are unique, independent of $b$, when restricted to $K$. This means that $\g_J(\a_1,\dots,\a_{p-r})/b$, treated as elements of the total ring of fractions of $R$, are determined by $\eta$ and $\a_1,\dots,\a_{p-r}\in K$.
\end{theorem}

Recall that the \emph{total ring of fractions} of the ring $R$ is the localization $\fracR:=S^{-1}R$ of $R$ with respect to the multiplicative set $S$ of all \nzds in $R$. In this case the natural ring homomorphism $a\in R\mapsto a/1\in \fracR$ is injective. 
(In particular, if $R$ is an integral domain then $S^{-1}R$ is a subring of the field of fractions of $R$.) 
Any free $R$-module $M\simeq R^m$ has a similar embedding into its localization $\fracM:=S^{-1}M\simeq (\fracR)^m$. The free dual module $\fracM^*$ is naturally isomorphic to the localization $S^{-1}M^*$ of $M^*$. Finally, the inclusion $K\subset M^*$ extends to a unique inclusion of the localizations $\fracK:=S^{-1}K\subset S^{-1}M^*$. This implies the following

\begin{remark}\label{rem-ring of fractions}\rm 
The forms $\g_J/b$ can be treated as elements of $S^{-1}(\La^{p-r}M)\simeq \La^{p-r}\fracM$. We can rephrase assertion (ii) of the theorem as uniqueness of the $(p-r)$-forms $\g_J/b$ restricted to $\fracK$.
\end{remark}

The theorem immediately follows from the following more general lemma. Let $b'$ and $b''$ be \nzds in $R$ and let $S_{[b',b'']}:=\{(b')^m (b'')^n\}_{m,n\ge 0}$ be the multiplicative set of powers of $b'$ and $b''$. 

\begin{lemma}
If
\[
b'\eta=\sum_{J\in\JJ(r,k)}\o_J\w\g'_J \quad \text{and} \quad b''\eta=\sum_{J\in\JJ(r,k)}\o_J\w\g''_J,  
\]
then for all $\a_1,\dots,\a_{p-r}\in K$ we have $\g'_J(\a_1,\dots,\a_{p-r})/b'=\g''_J(\a_1,\dots,\a_{p-r})/b''$ in the localization $S^{-1}_{[b',b'']}R$ of $R$.
\end{lemma}
 
\begin{proof}
Multiplying both sides of the former equality by $b''$ and of the latter equality by $b'$ and subtracting one from the other we get
\[
0=\sum_{J\in\JJ(r,k)}\o_J\w\g_J,\ \ \text{where}\ \ \g_J=b''\g'_J-b'\g''_J. \eqno{(U)}
\]
We will prove that $\g_J\vert_K=0$. Given $\a\in M^*$, we will use the operator of interior product $\a\rfloor:\La^qM\to \La^{q-1}M$ defined by $\a\rfloor\g=\g(\a,\cdot,\dots,\cdot)$, where $\g$ is treated as $q$th skew-symmetric form on $M^*$. 
 
Consider arbitrary elements $\a_1,\dots,\a_{p-r}\in K$. Since $\a_i\rfloor \o_j=0$, we have
\[
\a_{p-r}\rfloor\cdots\a_1\rfloor(\o_J\w\g_J)=(-1)^{r(p-r)}g_J\,\o_J
\]
where $g_J=\g_J(\a_1,\dots,\a_{p-r})$ belongs to $R$. Applying the operator $\a_{p-r}\rfloor\cdots\a_1\rfloor$ to both sides of (U) gives then
\[
\sum_{J\in\JJ(r,k)}g_J\o_J=0.
\]
We claim that this equality and the assumption that $I(\O)$ contains a \nzd imply that $g_J=0$ for all $J$. To see this let us fix $J\in\JJ(r,k)$ and multiply both sides in (U) by $\o_{J'}$, where $J'\in \JJ(k-r,k)$ is a complement of $J$ in $\JJ(k,k)$ so that $J\cap J'=\emptyset$ and $\o_J\w\o_{J'}=\pm\O$. We obtain $g_J\,\O=0$ since $\o_{J'}\w\o_{\tilde J}=0$ for all $\tilde J\in \JJ(r,k)$ such that $\tilde J\not=J$ (as at least one $\o_i$ is repeated in $\o_{J'}$ and $\o_{\tilde J}$). As $I(\O)$ contains a \nzd the equality $g_J\O=0$ implies that $g_J=0$. We conclude that $\g_J(\a_1,\dots,\a_{p-r})=0$ for any $\a_1,\dots,\a_{p-r}\in K$. This equality and the definition of $\g_J$ give that $b''\g'_J(\a_1,\dots,\a_{p-r}) =b'\g''_J(\a_1,\dots,\a_{p-r})$. Interpreting the latter equality in the localization ring $S^{-1}_{[b',b'']}R$ implies the result.
\end{proof}

Theorem \ref{thm1-existence} can be restated as a property of a free submodule $N\subset M$ of the free module $M$. Assume that $N$ has finite rank $k$ and let $\o_1,\dots,\o_k\in N$ be a basis in $N$.  As before, denote $\O=\o_1\w\cdots\w\o_k$ and let $I(N)$ be the ideal generated by the coefficients of $\O$ once a basis in $M$ and the corresponding basis in $\La^kM$ are chosen. Clearly, $I(N)$ is independent of the choice of the bases in $N$ and $M$. Let $''\owedge''$ denote the skew-tensor product of modules. 

\begin{theorem}\label{thm2}
Let  $1\le p,r,s\le m$ be fixed integers such that $p\ge r$ and $r+s=k+1$.  Consider an element  $\eta\in \La^p M$.

(i) Assuming that
\[
p-r\le \dpt I(N)-2  \eqno{(DC)}
\]
we have
\[
\eta\in \La^rN\owedge\La^{p-r}M \eqno{(A)}
\]
if and only if
\[
\La^sN\w\eta=0. \eqno{(B)}
\] 

(ii) 
Without assuming (DC), if $\eta$ satisfies (B) then there exists $n>0$ such that for any $a\in I(N)$, 
\[
a^n\eta \in \La^rN\owedge\La^{p-r}M. \eqno{(A')}
\]
Here the exponent $n$ is independent of $\eta$ if the module of $\eta$ satisfying (B) is finitely generated.
Conversely, if (A') holds for some \nzd $a\in R$ then (B) holds, too.
\end{theorem}

\begin{remark}\rm
As in Theorem \ref{thm1-existence}, if the module of $\eta$ satisfying (B) is finitely generated then the exponent $n$ in statement (ii) can bo chosen independent of $\eta$ (e.g., it is so if $R$ is Noetherian or condition (B) is trivial, i.e., $p+s>m$).
\end{remark}

Note that if for $\o_1,\dots,\o_k\in M$ the ideal $I(\O)$ contains a \nzd then the submodule $N=\spn_R\{\o_1,\dots,\o_k\}$ of $M$
is free and $\o_1,\dots,\o_k$ is its basis. (If $\sum_ia_i\o_i=0$ then multiplying both sides by $\o_1\w\cdots\hat \o_j\cdots\w\o_k$, with $\o_j$ omitted, we get $a_j\O=0$ and, as there is a \nzd in $I(\O)$, we get $a_j=0$, for all $j$.)

\begin{proposition}
Theorem \ref{thm1-existence} implies Theorem \ref{thm2}. Conversely,  Theorem  \ref{thm2} implies Theorem \ref{thm1-existence} if in the latter we assume in statement (ii) that $I(\O)$ contains a non-zero-divisor. 
\end{proposition}

\begin{proof}
To prove that Theorem \ref{thm1-existence} implies Theorem \ref{thm2} note that the assumption that $N$ is free of rank $k$ means that it has a basis $\o_1,\dots,\o_k\in M$. Any two such bases are related by an invertible $k\times k$ matrix over $R$. Consequently, whatever basis $\o_1,\dots,\o_k$ of $N$ we choose the modules $\La^qN$, $q\le k$, are spanned by the elements $\o_I$ with $I\in \JJ(q,k)$. Thus, conditions (A), (A'), and (B) in Theorem \ref{thm2} are equivalent to the corresponding conditions in Theorem \ref{thm1-existence}. 

To show the converse implication one should use the elements $\o_1,\dots,\o_k\in M$ in Theorem \ref{thm1-existence} to define the submodule $N\subset M$ generated by them. Such submodule $N$ is free which follows from the fact that $I(\O)$ contains a non-zero-divisor. In statement (i) of Theorem \ref{thm1-existence} the depth condition (DC) implies that $\dpt I(\O)\ge 2$ and there is a \nzd in $I(\O)$. In statement (ii) the existence of such \nzd is explicitely assumed. Once $N$ is free the statements of Theorem \ref{thm2} imply the corresponding statements in Theorem \ref{thm1-existence} since conditions (A), (A'), and (B) in both theorems are equivalent.
\end{proof}

\section{Residua}

The results from the preceding section may get a new meaning if the ring $R$ and the free $R$-module $M$ are replaced by their quotients. One such meaning is exploited in Sections 4 and 5.
Here we only mention that if $R$ is a ring of functions or function germs on a manifold and $I\subset R$ is a suitable ideal, the quotient $R/I$ represents the ring of functions or function germs on the set of zeros of $I$. In this case the forms $\g_J$ in Theorem  \ref{thm1-existence}, unique in the sense of Theorem \ref{thm1-uniqueness}, become unique residues of the original form $\eta$ on the set of zeros of $I$. 

As a preparation for the next two sections we restate main results from Section 3 for the ring $R$ replaced by its quotient by a finitely generated ideal.  
For precision, given a proper ideal $\II\subset R$ we introduce the quotient ring and the quotient module
\[
\bR=R/\II, \quad \bM=M/\II M.
\]
The $\bR$-module $\bM$ is again free of rank $m$ having a basis $\{\be_i\}$ induced by the basis $\{e_i\}$ in $M$. There are canonical homomorphisms $R\to\bR$, $M\to\bM$ and the image of an element of $R$ or $M$ will be marked with a bar, in particular $\o_i\mapsto \bo_i=\o_i+\II M$.  As the exterior product in $\La M$ induces an exterior product in $\La\bM$, again denoted by $''\w''$, we can introduce the product of all $\bo_i$
\[
\bO=\bo_1\w\cdots\w\bo_k
\]
which is an element of $\La^k\bM$. We denote by $I(\bO)$ the ideal in $\bR$ generated by the coefficients of $\bO$ written in the natural basis induced in $\La^k\bM$ by the basis $\be_1,\dots,\be_m$. 

Theorems \ref{thm1-existence} and  \ref{thm1-uniqueness}, together,  can be restated for the quotient ring $\bR$ and the quotient modules $\bM$ and $\La^q\bM\simeq\La^qM/\II\La^qM$ in the following form. Below, in conditions (A), (B) and (A'), we write elements of the quotient modules as equivalent classes in the original modules and assume that the ideal $\II$ is generated by elements $f_1,\dots,f_\ell\in R$,
\[
\II=f_1R+\cdots +f_\ell R.
\]

\begin{theorem}\label{thm-residua}
Fix a sequence $\o_1,\dots,\o_k\in M$ and integers $1\le p,r,s\le m$ such that $p\ge r$ and $r+s=k+1$.  

(i) Assume  that
\[
p-r\le \dpt I(\bO)-2.  \eqno{(DC)}
\]
Then an element  $\eta\in \La^p M$ can be represented in the form
\[
\eta=\sum_{J\in \JJ(r,k)}\o_J\w\g_J+\sum_i f_i\xi_i, \eqno{(A)}
\]
for some $\g_J\in \La^{p-r}M$ and $\xi_i\in\La^pM$, if and only if for all $I\in \JJ(s,k)$
\[
\o_I\w\eta=\sum_i f_i\b_{I,i} \eqno{(B)}
\]
for some $\b_{I,i}\in\La^{p+s}M$. 

(ii) Without assuming the depth condition (DC), for any $\eta\in \La^pM$ satisfying (B) there exists $n>0$ such that for any $b$ of the form $b=a^n$,  $a\in I(\Omega)$, we have
\[
b\eta=\sum_{J\in\JJ(r,k)} \o_J\w\g_J+\sum_i f_i\xi_i,  \eqno{(A')}
\]
for some $\g_J\in \La^{p-r}M$ and $\xi_i\in\La^pM$. Conversely, if (A') holds for some $b\in R$ such that its equivalence class $\bar b$ is a \nzd in $\bR$ then (B) holds, too.

(iii)  The elements $\g_J$ in the formula (A) are unique, modulo $\sum_i f_i\La^{p-r}M$, when treated as skew-symmetric $(p-r)$-linear forms on $M^*$ restricted to $K=\bigcap_i\ker \o_i\subset M^*$.  

(iv) More generally, if $b$ in the representation (A')  is an arbitrary element of $R$ such that the equivalence class $\bar b$ is a \nzd in $\bR$, then the elements $\bar\g_J/\bar b$ are uniquely determined by $\eta$ (independently of $b$) when treated as skew-symmetric $(p-r)$-linear forms on $\bM_{fr}^*$ restricted to $\bK_{fr}:=\bigcap_i\ker \bar\o_i\subset\bM_{fr}^*$, where $\bar\o_i\in \bM\subset \bM_{fr}$ are treated as elements of $\bM_{fr}$. Here  $\bM_{fr}=\bS^{-1}\bM$ denotes the localization of $\bM$ with respect to the multiplicative set $\bS$ of \nzds in $\bR$ and $\bM_{fr}^*$ is its dual module over $\bR_{fr}=\bS^{-1}\bR$ (cf. notation in Remark \ref{rem-ring of fractions}).  
\end{theorem}

\begin{proof} 
Statements (i) and (ii) in the theorem can be deduced from the corresponding statements in Theorem \ref{thm1-existence} where the ring $R$ should be replaced by the quotient ring $\bR=R/\II$ and the modules $M$ and $\La^qM$ by the quotient modules $\bM=M/\II M$ and $\La^q\bM\simeq\La^qM/\II\La^qM$. Namely, one easily verifies that conditions (A), (A') and (B) in Theorem \ref{thm1-existence}, stated for the quotient modules, are equivalent to the corresponding conditions in Theorem \ref{thm-residua}. Thus the first two statements in Theorem \ref{thm-residua} are direct consequences of Theorem  \ref{thm1-existence}.

Statements (iii) and (iv) follow analogously from Theorem \ref{thm1-uniqueness} where the original modules should be replaced with the quotient ones. 
\end{proof}

The unique elements $\Res_J(\eta):=(\bar\g_J/\bar b)\vert_{\overline{K}_{fr}}$ can be called \emph{algebraic residua} of $\eta$ with respect to $\o_1,\dots,\o_k$ and $\II\subset R$.

Consider the ideal  $I(f_1,\dots,f_\ell,\O):=\spn_R\{\II,I(\O)\}$ in $R$ generated by the generators $f_1,\dots,f_\ell$ of $\II$ and the coefficients of the $k$-form $\O$.
The depth of the ideal $I(\bO)$ in Theorem \ref{thm-residua}  can be computed from the ensuing relation.

\begin{proposition}\label{prop1}
If $R$ is Noetherian and the sequence $f_1,\dots,f_\ell$ is regular then $\dpt\II=\ell$  and
\[
\dpt I(f_1,\dots,f_\ell,\O)=\dpt I(\bO)+\ell.
\]
\end{proposition}

\begin{proof} 
The proof is an exercise on regular sequences and depth. The equality $\dpt\II=\ell$ is a consequence of regularity of the sequence $f_1,\dots,f_\ell$ and the inequality $\dpt I \le \gen I$ true for any proper ideal in Noetherian rings  (Property 9 in Appendix).  
The displayed equality follows from the following property: if $I$ is a proper ideal of a Noetherian ring $R$ then any regular sequence in $I$ can be completed to a maximal regular sequence in $I$ of length equal to $\dpt I$ (Property 7 in Appendix). Taking $I= I(f_1,\dots,f_\ell,\O)$ and the regular sequence $f_1\cdots,f_\ell$ in $I$, we can complete it to a maximal regular sequence $f_1\cdots,f_\ell,\cdots,f_q$ with $q=\dpt I$. Then $f_{\ell+1}+\II,\dots,f_q+\II$ is a regular sequence in $I(\bO)$ and thus we have $\dpt I\le \ell+\dpt I(\bO)$. 

Vice versa, let $\bar d =\dpt I(\bO)$ and consider a regular sequence $f_{\ell+1}+\II,\dots,f_{\ell+\bar d}+\II$ in $I(\bO)$ with $f_i\in I(\O)$. Then the sequence $f_1,\dots,f_\ell,\dots,f_{\bar d}$ of elements of $I$ is regular and we get the converse inequality $\dpt I\ge \ell+\dpt I(\bO)$.
\end{proof}

We separately state the special case of $r=k$ in Theorem \ref{thm-residua} which will have direct consequences related to logarithmic residua discussed in the next section.  

\begin{corollary}\label{cor-geometric residua}
(i) Assume that $p\ge k$ and
\[
p\le \dpt I(\bO)+k-2.  \eqno{(DC)}
\]
Then an element  $\eta\in \La^p M$ can be represented in the form
\[
\eta=\o_1\w\cdots\w\o_k\w\g+\sum_j f_j\xi_j, \eqno{(A)}
\]
for some $\g\in \La^{p-k}M$ and $\xi_j\in\La^pM$, if and only if
\[
\o_i\w\eta=\sum_j f_j\b_{i,j} \ \ \forall\  i\in\{1,\dots,k\} \eqno{(B)}
\]
for some $\b_{i,j}\in\La^{p+1}M$. 

(ii) If $p\ge k$ and $\eta\in \La^pM$ satisfies condition (B) then, without assuming (DC), there exists $n>0$ such that for any $b=a^n$, with $a\in I(\Omega)$, we have 
\[
b\eta=\o_1\w\cdots\w\o_k\w\g+\sum_j f_j\xi_j, \eqno{(A')}
\]
for some $\g\in \La^{p-k}M$ and $\xi_j\in\La^pM$. Conversely, if (A') holds with $b\in R$ such that its equivalence class $\bar b$ is a \nzd in $\bR$ then (B) holds, too.

(iii) If $b$ in formula (A') is an arbitrary element of $R$ such that its equivalence class $\bar b$ is a \nzd in $\bR$, then the $(p-k)$-form $\g$ on $M^*$ is unique, modulo $\sum_j f_j\La^{p-k}M$, when restricted to $K=\bigcap_i\ker \o_i\subset M^*$.  Moreover, in terms of the total ring of fractions $\bR_{fr}$ of the ring $\bR=R/\II$, and the corresponding module $\bM_{fr}$ and its dual $\bM_{fr}^*$,  the $(p-r)$-form $\bar \g/\bar b$ on $\bM_{fr}^*$ is independent of the choice of $b$, when restricted to $\bK_{fr}=\bigcap_i\ker \bar\o_i\subset \bM_{fr}^*$.
\end{corollary}

\begin{example}\label{rem-residua}{\rm
Let $R=\OO_0(\C^n)$ be the ring of holomorphic function germs at $0\in\C^n$ and let $f_1,\dots,f_\ell$ be its elements. Consider elements $\o_1,\dots,\o_k\in M=R^m$ which can be treated as germs at $0$ of  holomorphic maps $\C^n\to \C^m$. More geometrically, if $\C^n$ is replaced by a complex manifold $X$ of dimension $n$ and $\C^n\times\C^m$ is replaced by a holomorphic vector bundle $E\to X$ with fiber $F\simeq \C^m$ then we can think of $R$ as the ring of holomorphic function germs at a point $x\in X$ and of $\o_i$ as germs at $x$ of holomorphic sections of $E$. Then $M$ is isomorphic to the module of germs at $x$ of holomorphic sections of the bundle $E$ and $\La^qM$ is isomorphic to the $R$-module of germs at $x$ of holomorphic sections of the bundle $\La^qE$ which is the $q$th exterior product of the bundle $E$. 
In this case the second term on the right hand side of (A) and (A') vanishes at points of the set $Z=\{f_1=\cdots=f_\ell=0\}$  of common zeros of $f_j$. At such points the algebraic residue  $\g\vert_K+\sum_jf_j\La^{p-k}M\vert_K$ has the geometric interpretation of the field of the skew-symmetric $(p-k)$-forms $\g(x)\in \La^{p-k}E_x$ which are uniquely defined on the common kernel $K_x=\cap_j\ker\o_j(x)\subset E^*_x$ at points in $Z$ where $\o_1(x),\dots,\o_k(x)$ are linearly independent. A detailed discussion of the case where $\ell=k$, $E$ is the holomorphic cotangent bundle of a complex manifold $N$ and $\o_j=df_j$ will be given in Section 5.}
\end{example}

\section{Logarithmic residua}

In algebraic and complex analytic geometry a special role is played by the notion of logarithmic residue. Generally, this is a differential form defined on the divisor (pole) of a meromorphic form. Here we define meromorphic forms and their residua in a purely algebraic way as a special case of the algebraic residue from the preceding section. One of our motivations is to make it sufficienly general so that it is applicable in certain problems in differential geometry where appearence of singularities is natural. 

As earlier, $R$ will denote a unital commutative ring. Recall that a derivation on $R$ is a map $X:R\to R$ satisfying 
\[
X(f+g)=X(f)+X(g),
\]
\[
\qquad\ X(fg)=X(f)g+fX(g),
\] 
for all $f, g\in R$. The set of all derivations of $R$ is an $R$-module, denoted $\Der(R)$. 

Let $M=\Der(R)$ and assume that it is a free module of finite rank $m$. 
(This holds e.g. when $R$ is the ring of $C^r$  function germs on $m$-manifold of class $C^r$, where $r\in\{\infty,\omega,\text{holomorphic}\}$, in which case $M=\Der(R)$ is the modules of vector field germs of class $C^r$.)  Any element $f\in R$ defines a homomorphism $df:M\to R$ defined by
\[
df(X)=X(f)
\]
and called here \emph{algebraic differential} of $f$. 

Before stating the main theorem in this section we reformulate Corollary \ref{cor-geometric residua} to a setting which has immediate differential-geometric interpretations when $R$ is the ring of function germs of class $C^r$ as mentioned above. Let us fix a sequence of elements $f_1,\dots,f_k\in R$ generating a proper ideal $\II\subset R$. Consider the corresponding homomorphisms 
\[
df_1,\dots,df_k:M\to R
\]
which are elements of the dual free module, $df_i\in M^*=\Der(R)^*$.  We replace the module $M$ in Corollary \ref{cor-geometric residua} by its dual $M^*$ and take \[
\o_1=df_1,\ \dots\ ,\o_k=df_k.
\]
Consider the quotient ring $\bR=R/\II$ and the quotient $\bR$-module
\[
\bM^*=M^*/f_1M^*+\cdots+f_k M^*
\]
which is again free and has rank $m$. Let $\overline{df}_i=df_i+\II M^*\in \bM^*$ denote  the quotient counterparts (equivalence classes)  of the elements $df_i$.  Denote
\[
\bO=\overline{df}_1\w\cdots\w \overline{df}_k
\]
which is an element of $\La^k\bM^*$. Let $I(\bO)$ denote the ideal in $\bR$ generated by the coefficients of $\bO$ written in the basis $\be_{i_1}\w\cdots\w\be_{i_k}$, $i_1<\cdots<i_k$, where $\be_1,\dots,\be_m$ is a basis in $M^*$.

\begin{corollary}\label{cor-logarithmic residua}
Assume the set $M=\Der(R)$ of derivations of $R$ is a free module of rank $m$ and $f_1,\dots,f_k\in R$ are fixed elements generating a proper ideal in $R$. 

(i) Let $p\ge k$ and assume that 
\[
p\le \dpt I(\bO)+k-2.  \eqno{(DC)}
\]
Then an element  $\eta\in \La^p M^*$ can be represented in the form
\[
\eta=df_1\w\cdots\w df_k\w\g+\sum_j f_j\xi_j, \eqno{(A)}
\]
for some $\g\in \La^{p-k}M^*$ and $\xi_j\in\La^pM^*$, if and only if for all $i\in\{1,\dots,k\}$
\[
df_i\w\eta=\sum_j f_j\b_{i,j} \eqno{(B)}
\]
for some $\b_{i,j}\in\La^{p+1}M^*$. 

(ii) Without assuming the depth condition (DC), given any $\eta\in \La^pM^*$ satisfying (B) there exists $n>0$ such that for arbitrary $b=a^n$ with $a\in I(\Omega)$ we have
\[
b\eta=df_1\w\cdots\w df_k\w\g+\sum_j f_j\xi_j \eqno{(A')}
\]
for some $\g\in \La^{p-k}M^*$ and $\xi_j\in\La^pM^*$. Conversely, if (A') holds for some $b\in R$ such that $\bar b$ is a \nzd in $\bR$ then (B) holds, too.

(iii) If, in the representation (A'), $\bar b$ is a \nzd in $\bR$ or $b=1$ as in (A), then the skew-symmetric $(p-k)$-linear form $\g$ on $M$ is unique, modulo $\sum_j f_j\La^{p-k}M^*$, when restricted to $K=\bigcap_i\ker df_i\subset M$  (it depends on $\eta$ and $b\in R$).
\end{corollary}

The corollary is a direct consequence of Corollary \ref{cor-geometric residua}. The reader may formulate an analogous corollary to Theorem \ref{thm-residua}.
The element 
\[
\bg\vert_K:=\g\vert_K+\sum_j f_j\La^{p-k}M^*\vert_K,
\]
with $\g$ appearing in (A), can be called \emph{logarithmic residue} of $\eta$ with respect to $f_1,\dots,f_k\in R$, consistently with the next theorem. Note that if there is a repetition, up to a unit, in the sequence of $f_1,\dots,f_k$, for example $f_i=gf_j$ for a pair $i\not=j$, then $df_i=gdf_j+f_jdg$ and $\bO=0$. In this case statement (i) is empty, as condition (DC) fails, and statement (ii) is trivial as $a=0$.

In order to state the main result on logarithmic residua of meromorpic foms we clarify the terminology.
As earlier (cf. Remark \ref{rem-ring of fractions}), we embed the ring $R$ into its total ring of fractions $R_{fr}$ and, analogously, the modules $M$ and $M^*$ are canonically embedded into their localizations $M_{fr}$ and $M_{fr}^*$ over the multiplicative set $S$ of \nzds in $R$. An element $\o=\eta/c\in \La^q\fracM^*$, $\eta\in \La^q M^*$, $c\in S$, will be called \emph{meromorphic form} with \emph{divisor} $c$, if $c$ is not a unit and there is no other such representation of $\o$ with invertible $c$.  
The forms $\o$ which can be represented as $\o=\eta/1\in \La^q\fracM^*$ will be called \emph{regular}. Assume that the divisor $c$ can be factorized, $c=f_1\cdots f_k$, where all $f_i\in R$ are prime. Then the meromorphic form can be written as
\[
\o=\frac{\eta}{f_1\cdots f_k}. \eqno{(MF)}
\]
We will call such form of  $\o$ \emph{reduced} if there are no repeated (up to units) factors among the primes $f_i$ and $\eta$ is not divisible over any of $f_i$. The correspondig divisor $c=f_1\cdots f_k$ is then called \emph{reduced}. A meromorphic form will be called \emph{reducible} if it can be represented in the reduced form (MF). If $R$ is a unique factorization domain then every meromorphic form in $\La^q\fracM^*$ is reducible and its reduced divisor $c=f_1\cdots f_k$ is unique up to a unit (then $k$ is unique and the factors $f_1,\dots,f_k$ are unique up to order and units).   

Given the factors $f_1,\dots,f_k$ of the divisor in (MF) we can use the earlier notation with the ideal $\II=f_1R+\cdots+f_kR$, the factor ring $\bR=R/\II$, the factor modules $\bM=M/\II M$, $\bM^*=M^*/\II M^*$, and the elements $\overline{df}_i\in\bar M^*$, where  $\overline{df}_i:\bar M\to \bR$ are defined on the equivalence classes by $\overline{df}_i(v+\II M)=df_i(v)+\II$. We assume that $M=\Der(R)$ is a free module of rank $m\ge k$ and denote by $\hf_i$ the product of $f_1,\dots,f_k$ with omitted $f_i$. 

\begin{theorem}\label{thm-logarithmic residua}
Assume that $p\ge k$ and $\o\in \La^p\fracM^*$ is a reducible meromorphic form (MF) with a reduced divisor $f_1\cdots f_k$ such that its factors generate a proper ideal in $R$. Then the following holds.

(i) If $\o$ satisfies the condition
\[
df_i\w\o=0\ \mod\ \sum_j (\hat f_j)^{-1}\La^{p+1}M^*,\ \ i=1,\dots,k,
\]
then there exists $n\ge 0$ such that for $b=a^n$, where $a$ is any element of the ideal $I(df_1\w\cdots\w df_k)$ of $R$, the form $b\,\o$ can be written as
\[
b\,\o=\frac{df_1}{f_1}\w\cdots\w\frac{df_k}{f_k}\w\g+\sum_i\frac{\xi_i}{\hat f_i}  \eqno{(LR)}
\]
for some regular forms $\g\in\La^{p-k}M^*$ and $\xi_i\in\La^{p}M^*$. 

(ii) If in addition $\dpt I(\overline{df}_1\w\cdots\w \overline{df}_k)\ge p-k+2$ in the quotient ring $\bR$ then the representation (LR) holds with $b=1$.

(iii) If $R$ is a unique factorization domain and $b\,\o$ is of the form (LR) with $\bar b$ a \nzd in $\bR$, then the meromorphic $(p-k)$-form ${\bar b}^{-1}\bg\in\La^{p-k}\bM_{fr}^*$ restricted to $\bK_{fr}=\ker \overline{df}_1\cap\cdots\cap \overline{df}_k\subset\bM_{fr}$,
\[
\Res(\o):=\bar b^{-1}\bg\vert_{\bK_{fr}}, \eqno{(RES)}
\]
is, up to sign, uniquely defined by $\o$ (cf. Remark \ref{rem-unique gamma} below). 
\end{theorem}

As in complex analysis, one can call the skew-symmetric  $(p-k)$-form $\Res(\o):\bK_{fr}^{p-k}\to\bR$ defined in (RES) the \emph{logarithmic residue} of the meromorphic form $\o$ with the reduced divisor $f_1\cdots f_k$. 

Note that the logarithmic representation formula (LR) can be written in the more compact form
\[
b\,\o=d \log(f_1)\w\cdots\w d \log(f_k)\w\g  \mod\ \sum_i (\hat f_i)^{-1}\La^pM^*,  \eqno{(LR')}
\]
where we denote $d\log(f_i)=df_i/f_i$.  The divisors of the reduced versions of the forms $\xi_i/\hat f_i$ in the formulae (LR) are strictly ``weaker'' then $f_1\cdots f_k$ since $f_i$ is omitted and $\xi_i$ are regular. We can then say that the first ``logarithmic'' term in both formulae represents the most singular part of the form $b\,\o$.

{\it Proof of Theorem \ref{thm-logarithmic residua}.}
Having a meromorphic form $\o=\eta/(f_1\cdots f_k)$ we may consider its regular part $\eta$ and apply Corollary \ref{cor-logarithmic residua} to $\eta$. It is then evident that  
the first two statements of the theorem follow from statements (ii) and (i) in Corollary \ref{cor-logarithmic residua} by dividing both sides in equalities (A) and (A') by $f_1\cdots f_k$.

The uniqueness statement can be similarly deduced from statement (iii) in Corollary \ref{cor-logarithmic residua}. Since the formula (LR) implies that $\bar b^{-1}\bg$ is independent of the \nzd $\bar b$, it is enough to show that it is also independent of the choice of the reduced representation $\o=\eta/(f_1\cdots f_k)$ of $\o$. The assumptions that $R$ is a unique factorization domain and the divisor $f_1\cdots f_k$ is reduced imply that the freedom in choosing the reduced representation $\o=\eta/(f_1\cdots f_k)$ of a given $\o$ is restricted to multiplying $\eta$ and the factors $f_i$ by units $g$ and $g_i$, respectively, so that $g=g_1\cdots g_k$. Changing $f_i$ for $\tilde f_i= g_if_i$ gives that $d\tilde f_i/\tilde f_i=df_i/f_i+dg_i$ and the terms $dg_i$ eliminate the factor $f_i$ in the denominator of the first term in (LR). This means that they contribute to the second term in (LR) while the first term remains unchanged, thus $\bg$ remains unchanged.
\hfill$\Box$ 

\begin{remark}\label{rem-unique gamma}\rm 
Note that permuting the order of $f_1,\dots,f_k$ changes the sign of $\g$ in the formula (LR) by the signature of the permutation since $df_1\w\cdots\w df_k$ changes the sign in this way.
\end{remark}

\section{Multidimensional logarithmic residua in complex analysis}

One consequence of Theorem \ref{thm-logarithmic residua} from the preceding section is a result in complex geometry stated below. Let $N$ denote a complex manifold of dimension $m$ which, without loosing generality, can be taken $\C^m$. We denote by $R=\OO_x(N)$  the ring of holomorphic function germs at a point $x\in N$ (e.g. $x=0\in \C^m$).  It is well known that the ring $\OO_x(N)$ is a Noetherian unique factorization domain.  

Consider the $R$-module $\Vect_x(N)$ of germs at $x$ of holomorphic vector fields on the manifold $N$ and denote by $\La_x(N)$ the $R$-module of germs at $x$ of holomorphic one forms on $N$. Both modules are free of rank $m$ and are dual to each other by the natural pairing $\langle\o,X\rangle=\o(X)$ of $\o\in\La_xN$ and $X\in\Vect_x(N)$.  By $\La_x^q(N)$ we denote the $R$-module of germs at $x$ of holomorphic $q$-forms on $N$ which is the $q$th skew-symmetric product of  $\La_x(N)$. Germs at $x$ of considered objects can be replaced by their representatives defined in a small neighbourhood of $x$ and, if convenient, we will do this without mentioning. 
  
Fix $x\in N$ and consider function germs $f_1,\dots,f_k$ in $\OO_x(N)$ generating a proper ideal $\II\subset \OO_x(N)$. The set of its zeros
\[
Z_f=\{z\in N:\, f_1(z)=\cdots=f_k(z)=0\}
\]
is an analytic set germ at $x$. Let $df_i$ denote holomorphic differentials of $f_i$. 

Consider the holomorphic k-form germ $\O\in\La_x^k(N)$ defined pointwise as
\[
\O(z)=df_1(z)\w\cdots\w df_k(z).
\]
We will also use the set germ of common zeros of $f=(f_1,\dots,f_k)$ and $\O$,
\[
Z_{f,\O}=\{z\in N:\, f_1(z)=\cdots=f_k(z)=0,\ \O(z)=0\}.
\]

Recall that, given an open subset $U\subset N$  and an analytic subset $Z\subset U$, a point $z\in Z$ is called \emph{regular} if the intersection of $Z$ with a small neighbourhood $V\subset U$ of $z$ is a complex submanifold of $U$. Otherwise such a point is called \emph{singular}. The subset of regular points in $Z$ will be denoted $Z_{reg}$. The dimension of $Z$ at a regular point $z\in Z$ is just the dimension of the submanifold $Z\cap V$, denoted $\dim_zZ$. The dimension of $Z$ at a singular point $z\in S\subset Z$, also denoted $\dim_zZ$, is defined as the supremum of dimensions of $Z$ at regular points in $Z$ in a (arbitrarily) small neighbourhood of $z$. The codimension of $Z$ at $z\in Z$ is defined as $\codim_zZ=\dim N-\dim_zZ$.

Points in $Z$ which satisfy $\O(z)\not=0$ will be called \emph{strongly regular} and the set of such points will be denoted by $Z_{rreg}$. At points $z\in Z_{rreg}$ the differentials $df_1(z),\dots,df_k(z)$ are linearly independent. It then follows from the holomorphic version of implicit function theorem that in a neighbourhood $U$ of any such point $z\in Z_{rreg}$ the set of zeros $Z$ is a submanifold of codimension $k$.
The inclusions $Z_{rreg}\subset Z_{reg}\subset Z$ become equalities in each such neighbourhood. Moreover, we have
\[
T_zZ_{reg}=\ker df_1(z)\cap\cdots\cap\ker df_k(z), \ \ z\in Z_{rreg}. \eqno{(TS)}
\]

\begin{theorem}
\label{thm-complex logarithmic residua}
(i) Let $p\ge k$. Assume that $f_1,\dots,f_k$ is a regular sequence in $R=\OO_x(N)$ and
\[
p\le \codim_x Z_{f,\O}-2.  \eqno{(CD)}
\]
Then a holomorphic p-form germ $\eta\in \La_x^p(N)$ can be written as
\[
\eta=df_1\w\cdots\w df_k\w\g+\sum_j f_j\xi_j, \eqno{(A)}
\]
for some $\g\in \La_x^{p-k}(N)$ and $\xi_j\in\La_x^p(N)$, if and only if for all $i\in\{1,\dots,k\}$
\[
df_i\w\eta=\sum_j f_j\b_{i,j}, \ \ \text{with}\ \b_{i,j}\in\La_x^{p+1}(N).  \eqno{(B)}
\]

(ii) Without assuming the codimension condition (CD) there exists $n>0$ such that for any holomorphic p-form germ $\eta\in \La_x^p(N)$ satisfying (B) and any holomorphic function germ $b=a^n$, where $a\in I(\Omega)$, we have
\[
b\eta=df_1\w\cdots\w df_k\w\g+\sum_j f_j\xi_j, \eqno{(A')}
\]
for some $\g\in \La_x^{p-k}(N)$ and $\xi_j\in\La_x^p(N)$. Conversely, if (A') holds for some $b\in \OO_x(N)$ such that $\bar b=b+\II$ is a \nzd in the quotient ring $\OO_x(N)/\II$  then (B) holds, too.

(iii) The (p-k)-form $\g(z)$ in (A), as well as $\g(z)$ in (A') provided that $\bar b$ is a \nzd in $\bR$, is unique when restricted to the holomorphic tangent subspaces $T_zS_{reg}$ at strongly regular points $z\in Z_{rreg}$. 
\end{theorem}

Note that (A') can be transformed to formula (LR) in Theorem \ref{thm-logarithmic residua} with $\o=\eta/(f_1,\dots,f_k)$.

\begin{remark} \rm
It can be seen that statement (ii) implies Theorem 1 in Alexandrov \cite{A1} which is basic in his exposition of fundamentals of multidimensional complex residua theory. Statement (i) seems new.
\end{remark}

{\it Proof of Theorem \ref{thm-complex logarithmic residua}.}
Statements (i) and (ii) of the theorem follow from the corresponding statements in Corollary \ref{cor-logarithmic residua} if we take there $R= \OO_x(N)$ - the ring of holomorphic function germs, $M=\Vect_x(N)$ -  the $R$-module of germs at $x$ of holomorphic vector fields on $N$, and its dual $M^*=\La_x(N)$ - the module of germs at $x$ of holomorphic one forms on $N$. Then $\La^qM^*$ coincides with the $R$-module of germs at $x$ of holomorphic $q$-forms on $N$, $\La^qM^*=\La^q_x(N)$. The ring $\OO_x(N)$ is Noetherian thus the module of elements $\eta\in\La^p_x(N)$ satisfying condition (B) is finitely generated. Thus the exponent $n$ in $b=a^n$ in the formula (A') can be chosen independent of $\eta$, as in Theorem \ref{thm1-existence}. The depth condition (DC) in Corollary \ref{cor-logarithmic residua} is equivalent to the codimension condition (CD) above. This follows from the known fact in analytic geometry that the codimension of the set of zeros of a family of holomorphic functions germs in  $\OO_x(N)$ is equal to the depth of the ideal generated by these functions germs in  $\OO_x(N)$, see e.g. Chapter 5.3 in \cite{GH}. We obtain that $\codim_xZ_{f,\O}=\dpt I(f_1,\dots,f_k,\O)=k+\dpt I(\bO)$, where in the second equality we use Proposition \ref{prop1}. Thus the inequality $p\le \codim_xZ_{f,\O}-2$ holds if and only if $p-k\le\dpt I(\bO)-2$. This shows that Corollary \ref{cor-logarithmic residua} implies statements (i) and (ii) in Theorem \ref{thm-complex logarithmic residua}.

To show statement (iii) above suppose that we have two different representations (A') of the same $b\eta$ with $\g'$ and $\g''$ on the right side. Subtracting one from the other gives zero on the left side. Additionally, at points $z\in Z$ the functions $f_i$ vanish thus, denoting $\Delta\g=\g'-\g''$, we find that
\[
df_1(z)\w\cdots\w df_k(z)\w\Delta\g(z)=0,\ \ \ z\in Z.
\]
At strongly regular points $z\in Z_{rreg}$ the one forms $df_1(z),\dots,df_k(z)$ are linearly independent. Thus, on their common kernel $K(z)=\ker df_1(z)\cap\cdots\cap \ker df_k(z)$ the $(p-k)$-form $\Delta(z)$ vanishes (linear algebra). This proves that $\g'$ and $\g''$ restricted to the subspace $K(z)$ coincide at points $z\in Z_{rreg}$.  By formula (TS) we have $K(z)=T_zS_{reg}$, thus $\g'$ and $\g''$ coincide on $T_zS_{reg}$.
\hfill$\Box$

\section{Proof of Theorem \ref{thm1-existence}}

Throughout the proof $R$ denotes a unital commutative ring, $M$ is a free module over $R$ of rank $m$ and $\o_1,\dots,\o_k$ are fixed elements of $M$, where $k\le m$. Nonnegative integers $p,r,s$ satisfy the relations $m\ge p\ge r$ and $r+s=k+1$, as in Theorem \ref{thm1-existence}.The starting point is the following elementary version of statement (i) in Theorem \ref{thm1-existence}.

\begin{lemma}\label{lem1}
If the elements $\o_1,\dots,\o_k\in M$ can be completed to a basis in $M$ then for any $\eta\in\La^pM$ condition (B) in Theorem \ref{thm1-existence} implies condition (A).
\end{lemma}  

\begin{proof}
Let elements $\a_1,\dots,\a_{m-k}\in M$ complete $\o_1,\dots.\o_k$  to a basis of $M$. Then we have a basis of $\La^pM$ which consists of exterior products of the form
\[
\o_J\w\a_K=\o_{j_1}\w\cdots\w\o_{j_{\Jang}}\w\a_{k_1}\w\cdots\w\a_{k_{\Kang}},
\]
with  strictly monotone (possibly empty) multiindices $J$ and  $K$ such that their lengths  $\Jang$ and $\Kang$ satisfy $\Jang+\Kang=p$ (see the notation introduced in the beginning of Section \ref{Main results}).  Suppose that $\eta\in \La^pM$ is written as a linear combination of the elements of this basis,
\[
\eta=\sum_{J,K}a_{J,K}\,\o_J\w\a_K. \eqno{(E)}
\]
To prove that condition (A) holds it is enough to show that $\Jang\ge r$ in this expansion. Condition (B) states that $\o_I\w\eta=0$ for all strictly monotone multiindices $I$ of length $\Iang=s$. Assume first that $s+p\le m$ i.e., a priori, the form $\o_I\w\eta\in \La^{p+s}M$ can be nontrivial.  We have
\[
\o_I\w\eta=\sum_{J,K}a_{J,K}\,\o_I\w\o_J\w\a_K. \eqno{(E)}
\]
The products $\o_I\w\o_J\w\a_K$ are elements of a basis in $\La^{p+s}M$ whenever $\o_I$ and $\o_J$ do not contain the same $\o_i$, i.e., when $I\bigcap J=\emptyset$ with $I$ and $J$ treated as sets. Thus, condition $\o_I \w\eta=0$  implies that $a_{J,K}=0$ for $J$ such that $I\bigcap J=\emptyset$. Given $J$ with $\Jang<r$, we can always find $I$ with $\Iang=s$, $s+r=k+1$, such that $J\cup I=\emptyset$, thus $a_{J,K}=0$ for any $(J,K)$ with $\Jang<r$.  This means that all nonzero coefficients $a_{J,K}$ in the expansion (E) have $\Jang\ge r$, i.e. (A) holds.

Assume now that $p+s>m$. This, due to $r+s=k+1$, is equivalent to $p-r\ge m-k$.  Since the products $\a_K$ in (E) are nontrivial only when $\Kang\le m-k$, this gives that they are nontrivial only when $\Kang \le p-r$. Taking into account that $\Kang+\Jang=p$ we see that possible nontrivial terms in (E) satisfy $\Jang \ge r$ which implies condition (A).
\end{proof}

In the general proof we will use the ensuing special case of statement (i) of Theorem \ref{thm1-existence} which was proved in \cite{S} for $R$ Noetherian and in \cite{J}, Theorem 2.3, without this assumption.

\begin{lemma}\label{lem2}
If $p<\dpt I(\O)$ and $\eta\in\La^pM$ satisfies $\O\w\eta=0$ then $\eta=\sum\o_i\w\g_i$ for some  $\g_1,\dots,\g_k\in \La^{p-1}M$.
\end{lemma}  
 
We shall also need the following known fact (for an easy proof see e.g. \cite{J}, Lemma 2.6)

\begin{lemma}\label{lem3}
If a sequence $a_1,\dots,a_r\in R$ is regular then, for any $n\ge 1$, the sequence $a_1^n, a_2,\dots,a_r$ is also regular.
\end{lemma} 
  
{\bf Proof of Theorem \ref{thm1-existence}. Statement (ii).} 
We first prove the second statement which will be used for proving the first one. It is enough to prove that condition (A') holds  for the coefficients of $\O$. Indeed, if it holds for any such coefficient with some power $n$ then it holds for an arbitrary element of $I(\O)$ (which is a finite linear combinations of the coefficients) with a sufficiently large power depending on $n$ and on the length of the linear combination. 

Let $a$  be a coefficient of $\O$. We may assume that $a$ is not nilpotent, otherwise condition (A') trivially holds. 
Consider the multiplicative set of nonnegative powers of $a$, $\AA=\{a^i\}_{i\ge 0}$, where $a^0=1$. Let $\Rloc$ denote the localization of $R$ with respect to $\AA$. Elements of the ring $\Rloc$ can be represented as sums of ``fractions'' of the form $b/a^i$, $b\in R$. Similarly, let $\Mloc$ denote the localization of  $M$ with respect to $\AA$, with elements represented as sums of ``fractions'' of the form $m/a^i$, $m\in M$. Then $\Mloc$ is a module over $\Rloc$. Analogously, the modules $\La^pM$ can be localized with respect to the multiplicative set $\AA$ and these localizations are isomorphic to $\La^p\Mloc$. We have canonical homomorphisms $R\to \Rloc$ and $M\to \Mloc$ given by the transformations $c\mapsto [c]:=c/1$  and $\o\mapsto [\o]:=\o/1$. In particular, the basis $e_1,\dots,e_m$ of $M$ is transformed into the basis $[e_1],\dots,[e_m]$ of $\Mloc$ and the basis $e_I$, $I\in \JJ(p,m)$ of $\La^pM$ is transformed into the basis $[e_I]$ of $\La^p\Mloc$, where $[e_I]=[e_{i_1}]\w\cdots\w[e_{i_p}]$. 

The image $[a]$ of a coefficient $a$ of $\O$ under the homomorphism $R\to \Rloc$ is a coefficient of $[\O]=[\o_1]\w\cdots\w[\o_k]$ and is a unit in $\Rloc$. This implies that the elements $[\o_1],\dots,[\o_k]$ can be completed to a basis in $\Mloc$. Therefore, we can use Lemma \ref{lem1} for the elements $[\o_1],\dots, [\o_k]$ of $\Mloc$ and $[\eta]\in\La^p\Mloc$. Condition (B) satisfied for the elements  $\o_1,\dots, \o_k$ of $M$, $\eta\in\La^pM$ and $\o_I\in\La^sM$  implies that it is satisfied for the corresponding elements $[\o_1],\dots, [\o_k]$ of $\Mloc$, $[\eta]\in\La^p\Mloc$, and $[\o_I]=[\o_{i_1}]\w\cdots\w[\o_{i_s}]$. It follows from the lemma that 
\[
[\eta]=\sum_{J\in\JJ(r,k)}[\o_J]\w\tg_J,\ \ \tg_J\in\La^r\Mloc,
\]
where $[\o_J]=[\o_{j_1}]\w\cdots\w[\o_{j_r}]$. One can write $\tg_J=\hat\g_J/a^{n_1}$ for some $\hat\g_J\in\La^rM$ and $n_1\ge 0$, the same for all $J$.
Then, due to the definition of equality of elements in the localization $(\La^pM)_{[a]}\simeq\La^pM_{[a]}$, the above equality means that there exists $n_2\ge 0$ such that 
\[
a^{n_2}(a^{n_1}\eta-\sum_{J\in\JJ(r,k)}\o_J\w \hat \g_J)=0.
\]
This implies that equality (A') in Theorem \ref{thm1-existence} holds with $n=n_1+n_2$ and $\g_J=a^{n_2}\hat \g_J$.

The exponent $n$ found above depends on the choice of $\eta$. It can be taken independent of $\eta$ if the module of forms $\eta\in\La^pM$ satisfying condition (B) has a finite number of generators $\eta_1,\dots,\eta_\ell$. Namely, if $n_1,\dots,n_\ell$ are the exponents corresponding to the generators then it is evident from the formula (A') that the sum $n=n_1+\cdots+n_\ell$ is such a universal exponent.

\medskip

{\bf Statement (i).}
Since condition (A) trivially implies (B) (cf. Remark \ref{rem1}), it is enough to prove that condition (B) implies (A). We will use induction with respect to $1\le r\le \min\{k,p\}$, in decreasing order. The initial step of maximal $r$ is split into two subcases.

{\bf The case ${\bf r=k\le p}$.}
In this case we have $s=1$, thus condition (B) means that $\o_i\w\eta=0$ for all $1\le i\le k$. 
Let $a_1,\dots,a_d\in I(\O)$ be a regular sequence, where $d=\dpt I(\O)\ge 2$ by condition (DC). Then $a_1^n,a_2,\dots,a_d$ is also regular for any $n>0$, by Lemma \ref{lem3}. From statment (ii) it follows that there exists $n$ such that, for $b=a_1^n$,
\[
b\eta=\O\w\g,\ \ \text{for\ some}\ \ \g\in\La^{p-k}M. 
\]
Consider the quotient ring $\bR=R/bR$, the quotient module $\bM=M/bM$, and the modules $\La^q\bM\simeq\La^qM/b\La^qM$. Then we have natural homomorphisms $c\in R\mapsto \bc\in\bR$,   $\o\in M\mapsto \bo\in\bM$, and  $\eta\in \La^qM\mapsto \bar\eta\in\La^q\bM$. Applying these homomorphisms to the above equality gives
\[
0=\bO\w\bg,\ \ \text{with}\ \ \bg\in\La^{p-k}\bM,
\]
where $\bO=\bo_1\w\cdots\w\bo_k$. The sequence $\ba_2,\dots,\ba_d\in I(\bO)$ is regular in $\bR$, thus $\dpt I(\bO)\ge d-1$.  Additionally, $d-1>p-r$ by the depth assumption (DC). It follows then that $\dpt I(\bO)>p-r=p-k$. We can now apply Lemma \ref{lem2} for the quotient module, with $\bg$ playing the role of $\eta$, to deduce that
\[
\bg=\sum\bo_i\w\bg_i \ \ \text{with\ some}\ \ \bg_i\in\La^{p-k-1}\bM.
\] 
Lifting this equality to the original module gives
\[
\g=\sum\o_i\w\g_i+b\xi, 
\] 
with\ some $\g_i\in\La^{p-k-1}$ and $\xi\in\La^{p-k}M$. Plugging such $\g$ to the earlier formula $b\eta=\O\w\g$ gives
\[
b\eta-b\O\w\xi=\O\w\left(\sum\o_i\w\g_i\right) 
\]
and then
\[
b(\eta-\O\w\xi)=0, 
\]
since $\O\w\o_i=0$. Using the fact that $b=a_1^n$ was a \nzd we deduce that $\eta=\O\w\xi$, which was to be proved.

{\bf The case ${\bf r=p\le k}$.} 
By condition (DC) we have that $\dpt I(\O)\ge 2$, thus there is a regular sequence $a_1,a_2\in I(\O)$. It follows fom Lemma \ref{lem3} that $a_1^n,a_2$ is also a regular sequence, for any $n\ge 1$, and this fact will be used below. 

By statement (ii) of the theorem condition (B) implies condition (A') which means, in the case $p=r$, that there exist $n\ge 0$ and coefficients $b_J\in R$ such that 
\[
a_1^n\eta=\sum_Jb_J\o_J,
\] 
where $J\in\JJ(r,k)$. Denote $b=a_1^n$ and consider the quotient ring $\bR=R/bR$ and the quotient module $\bM=M/bM$. After passing to the quotients the above equality  reads 
\[
0=\sum_J\bb_J\bo_J,\  \  \bb_J\in \bR, \  \bo_J\in\La^r\bM,
\] 
with $\bo_J=\bo_{j_1}\w\cdots\w\bo_{j_r}$. For a given  $J\in \JJ(r,k)$ let $J'\in \JJ(k-r,k)$ denote the multiindex which is complementary to $J$, i.e., $J\bigcap J'=\emptyset$. Then multiplying both sides of the above equality by $\bo_{J'}$ gives 
\[
0=\bb_J\bo_J\w\bo_{J'}=\pm\bb_J\bO
\]
where $\bO=\bo_1\w\cdots\w\bo_k$. Since the ideal $I(\bO)$ contains the \nzd\ $\ba_2$ (the sequence $a_1^n,a_2$ is regular), we deduce that all $\bb_J=0$. This means that $b_J=bc_J=a_1^nc_J$, for some $c_J\in R$. Plugging such $b_J$ to the formula for $a_1^n\eta$ gives
\[
a_1^n\eta=a_1^n\sum_Jc_J\o_J,\  \  c_J\in R.
\] 
The fact that $a_1^n$ is a \nzd in $R$ implies that $\eta=\sum_Jc_J\o_J$, which was to be shown.

{\bf The case $\bf  r <\min\{p,k\}$.} 
To prove statement (i) for all $r<\min\{k,p\}$ assume it is true for $r+1$. Denote $d=\dpt I(\O)$ and let $a_1,\dots,a_d$ be a regular sequence in $I(\O)$ where, by our assumption, $d\ge p-r+2$.  By Lemma \ref{lem3} the sequence $a_1^n,a_2,\dots,a_d$ is also regular for any $n$. Condition (B) implies condition (A'), by statement (ii), thus for $\eta$ satisfying (B) we have
\[
a_1^n\eta= \sum_{J\in \JJ(r,k)}\o_J\w\g_J
\]
for some $n\ge 1$ and $\g_J\in \La^{p-r}(M)$.

Denote $b=a_1^n$. As earlier, we introduce the quotient ring $\bR=R/bR$ and the quotient modules $\bM=M/bM$, $\La^i\bM\simeq\La^iM/b\La^iM$. Under canonical homomorphisms elements of the original ring and of the corresponding modules have their canonical images, denotesd with bars, in the quotient ring and modules.  For  $\bO=\bo_1\w\cdots\w\bo_k$ we have  $\dpt I(\bO)\ge d-1$ since the sequence $\ba_2,\dots,\ba_d\in I(\bO)$ is regular. When replacing the elements in the above equality with their counterparts in the quotient objects we get zero on the left side, thus
\[
0= \sum_{J\in \JJ(r,k)}\bo_J\w\bg_J
\]
where $\bo_J=\bo_{j_1}\w\cdots\w\bo_{j_r}$ and $\bg_J\in\La^{p-r}\bM$.  Pick a multiindex $J\in \JJ(r,k)$ and let $J'\in \JJ(r',k)$ be its complement satisfying $r'+r=k$ and $J\bigcap J'=\emptyset$. Multiplying both sides of the above equality by $\bo_{J'}$ we obtain
\[
0=\bo_{J'}\w\bo_J\w\bg_J=\pm\bO\w\bg_J,
\]
since all other products in the sum vanish as they contain a repeated $\bo_i$.  We can now use Lemma \ref{lem2} with $\bg_J\in\La^{p-r}\bM$ playing the role of $\eta$ and $\O$ replaced with $\bO$. Namely, $p-r <\dpt I(\O)-1$, by the depth assumption (DC), and $d=\dpt I(\O)\le \dpt I(\bO)+1$, by the inequality mentioned earlier. Therefore the asumption $p-r<\dpt I(\bO)$  required in the lemma is satisfied and we deduce that
\[
\bg_J=\sum_i\bo_i\w\bg_{J,i}
\]
for some $\bg_{J,i}\in \La^{p-r-1}\bM$.  

The above equaliy can be written in the original modules as
\[
\g_J=\sum_i\o_i\w\g_{J,i}+b\xi_J,
\]
for some $\g_{J,i}\in \La^{p-r-1}M$ and $\xi_J\in \La^{p-r}M$. Plugging such $\g_J$ to the formula for $a_1^n\eta$ and taking into account that $b=a_1^n$ we obtain 
\[
a_1^n(\eta-\sum_J\o_J\w\xi_J)=\sum_J\sum _i\o_J\w\o_i\w\g_{J,i}.
\]
Multiplying both sides by $\o_{I'}$, with arbitrary $I'\in\JJ(k-r,k)$, we find that
\[
a_1^n(\eta-\sum_J\o_J\w\xi_J)\w\o_{I'}=0
\]
since each product $\o_J\w\o_i\w\o_{I'}$ vanishes as it has a repeated $\o_j$ for some $j\in\{1,\dots,k\}$. From the fact that $a_1^n$ is a \nzd we deduce that
\[
(\eta-\sum_J\o_J\w\xi_J)\w\o_{I'}=0,
\]
for any $I'\in \JJ(k-r,k)$. Thus the form $\eta-\sum_J\o_J\w\xi_J$ satisfies condition (B) with $s'=k-r$. 

Putting $r'=r+1$ we have $r'+s'=k+1$. Thus, we we can  use the induction assumption that statement (i) holds for $r+1$ and conclude that the element $\eta'=\eta-\sum \o_J\w\xi_J$ has a representation
\[
\eta-\sum_J\o_J\w\xi_J=\sum_{\hJ}\o_{\hJ}\w\hg_{\hJ}
\]
with $\hJ$ in $ \JJ(r+1,k)$ and  some $\hg_{\hJ}\in\La^{p-r-1}\bM$. Since $\o_{\hJ}$ are products of $r+1$ forms among $\o_1,\dots,\o_k$, the sum on the right side can be rearranged to a sum $\sum_J\o_J\w\widetilde{\g}_J$, with $J\in\JJ(r,k)$ and $\widetilde{\g}_J\in \La^{p-r}M$.
This implies that $\eta$ can be written in the form (A) and ends the proof.
\hfill $\Box$

%\section{Concluding remarks and open questions}

\section{Appendix: properties of the depth}

Let $R$ be a commutative ring with unit. Recall that a sequence of elements $a_1,\dots,a_q$ of $R$ is called {\it regular} if $(a_1,\dots,a_q)\not=R$ and $a_i$ is a non-zero-divisor on $R/(a_1,\dots,a_{i-1})$ for $i=1,\dots,q$ (in particular, $a_1$ is a non-zero-divisor on $R$). Here $(a_1,\dots,a_i)$ denotes the ideal in $R$ generated by the elements $a_1,\dots,a_i$. 

Given a proper ideal $I\subset R$, the \emph{depth} of $I$, denoted $\dpt I$, is the supremum of lengths of regular sequences in $I$. Additionally, one defines $\dpt R=\infty$. Below we list several properties which can be useful when verifying the depth condition (DC).
\begin{itemize}
\item[(1)] If  $I_1, I_2$ are ideals in $R$ and $I_1\subset I_2$ then $\dpt I_1\le \dpt I_2$ (trivial).
\item[(2)]  If $a_1,\dots,a_q\in I$ is a regular sequence then $\dpt I\ge \dpt I/(a_1,\dots,a_q)+q$.
\item[(3)] If for a fixed $i$ the sequences $a_1,\dots,a_{i-1},b,a_{i+1},\dots,a_r$ and $a_1,\dots,a_{i-1},c,a_{i+1},\dots,a_r$ are regular, and $a_i=bc$, then $a_1,\dots,a_{i-1},a_i,a_{i+1},\dots,a_r$ is regular. Vice versa, if $a_1,\dots,a_i,\dots,a_r$ is regular, with $a_i=bc$, and $(a_1,\dots,a_{i-1},b,a_{i+1},\dots,a_r)R\not=R$ then $a_1,\dots,a_{i-1},b,a_{i+1},\dots,a_r$ is regular. 
\item[(4)] A sequence $a_1,\dots,a_r$ is regular if and only if $a_1^{i_1},\dots,a_r^{i_r}$ is regular for given $i_1,\dots,i_r\ge 1$ (follows from (3)). Therefore, the depth of $I$ is equal to the depth of the radical of $I$.
\end{itemize}
For Propertites (3) and (4), see e.g. \cite{BJ}, Exercises in Chapter 7.1. 

For $R$ Noetherian we have  the following additional properties (cf. \cite{E} or \cite{BJ}, Chapter 7.1).
\begin{itemize}
\item[(5)] If $a_1,\dots,a_r$ is a maximal (with respect to inclusion) regular sequence in $I$ then $r=\dpt I$.
\item[(6)] Any regular sequence $a_1,\dots,a_q$ in a proper ideal $I\subset R$ can be completed to a maximal regular sequence $a_1,\dots,a_q,\dots,a_r$ in $I$.
\item[(7)] If $a_1,\dots,a_q\in I$ is a regular sequence then $\dpt I=\dpt(I/(a_1,\dots,a_q))+q$  (follows from (5) and (6)). 
\item[(8)] If $R$ is local then any permutation of a regular sequence is regular.
\item[(9)] The depth, the hight and the minimal number of generators of $I$, denoted $\gen I$, satisfy the inequalities $\dpt I\le \hgt I \le \gen I$.
\item[(10)] $\dpt I=\gen I$ if and only if $I$ is generated by a regular sequence.
\item[(11)] If $R$ is a Cohen-Macauley ring then $\dpt I=\hgt I$ for any ideal $I\in R$ (and vice versa).
\item[(12)] If $\o=(a_1\dots,a_r)\in R^r$ and the ideal $I$ in $R$ generated by $a_1,\dots,a_r$ is proper and nonzero then $\dpt I$ is equal to the maximal $p$ such that $i$th cohomology $H_i:\im \partial_{i-1}/\ker \partial_i$ in the Koszul complex $\partial_i:\La^iR^r\to \La^{i+1}R^r$ defined by $\eta\mapsto \o\w\eta$ vanishes for all $i< p$.
\end{itemize}

\end{document}